\documentclass{article}

\usepackage{graphicx}
\usepackage{amsfonts}
\usepackage{amssymb}
\usepackage{amsmath}
\usepackage{amsthm}
\usepackage{enumerate}

\newcommand{\X}{\mathfrak X}
\DeclareMathOperator{\Aut}{Aut}
\DeclareMathOperator{\Inn}{Inn}
\DeclareMathOperator{\m}{m}
\DeclareMathOperator{\sm}{sm}

\newtheorem{theorem}{Theorem}
\newtheorem{lemma}{Lemma}
\newtheorem{definition}{Definition}
\newtheorem{proposition}{Proposition}
\newtheorem{corollary}{Corollary}

\begin{document}

\title{The proper definition and Wielandt-Hartley's theorem for submaximal \( \X \)-subgroups
\let\thefootnote\relax\footnotetext{The first and third authors were supported by the program of
fundamental scientific researches of the Siberian Branch of Russian Academy of Science No. I.1.1, Project No. 0314-2019-0001.
The second author was supported by Russian Foundation for Basic Research, Project No. 18-31-20011.}\\
\medskip
\large To the 110th anniversary of Helmut Wielandt
}

\author{Danila~Revin \and Saveliy~Skresanov \and Andrey~Vasil'ev}

\date{}

\maketitle

\begin{abstract}
	A nonempty class \( \X \) of finite groups is called complete
	if it is closed under taking subgroups, homomorphic images and extensions.
	We deal with a classical problem of determining \( \X \)-maximal subgroups.
	We consider two definitions of submaximal \( \X \)-subgroups
	suggested by Wielandt and discuss which one
	better suits our task. We prove that these
	definitions are not equivalent yet Wielandt-Hartley's theorem
	holds true for either definition of \( \X \)-submaximality.
	We also give some applications of the strong version of Wielandt-Hartley's theorem.
	\medskip

	\medskip

\end{abstract}

\section{Introduction}

In this article we focus on the relationship between two definitions of
submaximal \( \X \)-subgroups of a finite group\footnote{We consider finite
groups only, and from now on the term ``group'' means a ``finite group''.}
given by H.~Wielandt: the first one appeared
in his lectures \cite{Wie4} delivered at T\"ubingen in 1963--64, and the
second one was presented in his talk~\cite{Wie3} at the celebrated
Santa Cruz conference on finite groups in 1979. We will show that these definitions are not
equivalent, yet Wielandt-Hartley's theorem for submaximal \( \X \)-subgroups
is true for either definition of submaximality. In its strong version this theorem
was announced by Wielandt in \cite{Wie3}, but the proof was never
published. As a demonstration of possible applications of the strong Wielandt-Hartley's
theorem, we prove a sufficient condition for conjugacy of submaximal
\( \X \)-subgroups in terms of projections into the factors of subnormal
series, obtain a characterization
of submaximal \( \X \)-subgroups in direct products and also find a new
criterion for subnormality.
The last section of the paper contains several short historical remarks.

We begin with the context where the notion of a submaximal \( \X \)-subgroup
arises. It is well known that one of the central topics in group
theory is a study of subgroups of a given group. Apart from arbitrary
subgroups, one can be interested in some special types of subgroups defined by
their arithmetic or group-theoretic properties, or in other words, by belonging to the
corresponding class \( \X \) of groups (abelian, nilpotent, solvable,
\( p \)-groups for a prime \( p \), \( \pi \)-groups for a set of primes \( \pi \), etc.).
This task is exceptionally difficult to achieve in its general setting, and
it is universally accepted that attention can be restricted to maximal subgroups, i.e.\ subgroups
which are maximal by inclusion
\begin{itemize}
	\item among proper subgroups, if we are interested in all subgroups;
		these subgroups are simply called \textit{maximal},
	\item or among \( \X \)-subgroups, i.e.\ subgroups from a class \( \X \);
		in this case we are talking about \textit{maximal \( \X \)-subgroups} or
		\textit{\( \X \)-maximal subgroups}.
\end{itemize}

Following Wielandt \cite{Wie4,Wie3}, we consider maximal \( \X \)-subgroups
only in the case of a so-called complete class \( \X \). A nonempty class \( \X \) of finite groups is said to be \textit{complete}
(``\textit{vollst\"{a}ndig}'' in Wielandt's terms \cite[Definition~11.3]{Wie4}),
if it is closed under taking subgroups, homomorphic images and extensions,
where the latter means that \( G \in \X \) whenever \( G \) contains a normal \( \X \)-subgroup \( A \) and \(G/A \in \X \).
Solvable groups, \( \pi \)-groups, and solvable \( \pi \)-groups,
where \( \pi \) is a set of primes, are examples of complete classes.
From now on the symbol \( \X \) will always mean a fixed complete class.

While studying a group \( G \), it is natural to deal with the factors of its composition series,
i.e.\ a subnormal\footnote{Recall that the
relation of normality between a group and its subgroup is not transitive
in the sense that if \( H \) is normal in \( G \) (we write \( H \unlhd G \))
and \( K \) is normal in \( H \), then \( K \) is not necessarily normal in \( G \).
A subgroup \( H \) of \( G \) is \textit{subnormal}
(we write \( H {\unlhd\unlhd} G \)) if there exists a series
\( G = G_0 \geq G_1 \geq \dots \geq G_n = H \), where \( G_i \unlhd G_{i-1} \) for
all \( i = 1, \dots, n \). In other words, subnormality is the transitive
closure of normality.
} series
\begin{equation}
 G=G_0\geq G_1\geq \dots\geq G_n=1,\label{series}
\end{equation}
whose factors \( G^i = G_{i-1}/G_i \) are simple groups. Recall that the
classical Jordan-H\"older theorem implies that the set of composition factors
is a group invariant, meaning that up to reordering and isomorphism it does
not depend on the series~(\ref{series}). The strategy of ``reduction to
simple groups'', which appeared at the dawn of group theory in works of Galois and Jordan,
became truly effective after the classification of finite simple groups
(CFSG) had been finished.

Applying this approach to our problem, define for a subgroup \( H \) of a group \( G \) the
\textit{projections}
\[ H^i=(H\cap G_{i-1})G_i/G_i \]
of \( H \) on the factors \( G^i \) of a subnormal series~(\ref{series}).
Wielandt \cite[(12.1)(b)]{Wie4} noticed an obvious fact that if all projections
\( H^i \) of \( H \) are maximal \( \X \)-subgroups of \( G^i \),
then \( H \) is a maximal \( \X \)-subgroup of~\( G \).
It is a far less trivial question if the converse of this statement holds.

As a positive example, consider the class \( \X \) of all \( p \)-groups for
some prime~\( p \). Suppose that \( H \) is a maximal \( \X \)-subgroup
of~\( G \). By the Sylow theorem, \( H \) is a Sylow
\( p \)-subgroup of \( G \), i.e.\ its order is equal to the highest
power of \( p \) dividing the order of \( G \). It can be easily shown
that if \( A \) is a normal subgroup of \( G \), then \( H \cap A \)
is a Sylow \( p \)-subgroup of \( A \) and \( HA/A \) is a Sylow
\( p \)-subgroup of~\( G/A \). As a consequence, \( H^i \) is a Sylow
\( p \)-subgroup of \( G^i \) for each~\( i \).

Does something similar hold for an arbitrary complete class \( \X \)?
More precisely, if \( \X \) is a complete class, \( A \) is a normal
and \( H \) is an \( \X \)-maximal subgroups of \( G \), is it true that
\begin{itemize}
	\item[(a)] \( HA/A \) is a maximal \( \X \)-subgroup of \( G/A \), and
	\item[(b)] \( H \cap A \) is a maximal \( \X \)-subgroup of \( A \)?
\end{itemize}

Wielandt showed that the answers to Questions (a) and (b) in general are negative.

In \cite[(14.2)]{Wie4}, it was demonstrated that there is a generic counterexample to Question~(a)
for every complete class \( \X \) with the following property: there is a group with nonconjugate maximal
\( \X \)-subgroups. Indeed, let \( N \) be such a group, and \( B \) an arbitrary group.
If \( G = N \wr B \) is the regular wreath product of \( N \) and \( B \) with the base subgroup \( A = N^{|B|} \),
then each (maximal or not) \( \X \)-subgroup of \( B = G/A \) is an image of some maximal \( \X \)-subgroup of~\( G \).

There are examples where the intersection of a maximal \( \X \)-subgroup
with a normal subgroup \( A \) is not \( \X \)-maximal in \( A \) (cf. Question (b)).
Such examples can be found even among almost simple groups\footnote{
Recall that a finite group \( G \) is called \textit{almost simple}
if its \textit{socle}, that is the subgroup generated by
all (nontrivial) minimal normal subgroups, is a nonabelian simple group.}
(see, e.g., \cite[p.~27]{Wie4} and
\cite[Tables 6 and 11]{GR3}). However, in contrast to the situation with
homomorphic images, not every \( \X \)-subgroup of a normal subgroup
\( A \) can be represented as an intersection of \( A \) and some
maximal \( \X \)-subgroup. The relevant constraint here is the theorem
proved by Wielandt \cite[Hauptsatz~13.2]{Wie4} and B.~Hartley
\cite[Lemmas~2 and 3]{Hart} independently, so we further refer to this result
and its variations as Wielandt-Hartley's theorems.

We write \( \m_\X(G) \) for the set of all maximal \( \X \)-subgroups of \( G \).
Recall that if \( P \) and \(Q \) are subgroups of a group \( G \) then the \textit{normalizer}
\[ N_Q(P) = \{ x \in Q \mid x^{-1}Px = P \} \]
of \( P \) in \( Q \) is the largest subgroup of \( Q \) which normalizes \( P \).

\begin{theorem}[Wielandt-Hartley's theorem for normal subgroups]\label{WieHart}
	Let \( G \) be a finite group and let \( \X \) be a complete class.
	If \( A \) is a normal subgroup of \( G \), then for every
	\( H \in \m_\X(G) \) the quotient \( N_A(H \cap A)/(H \cap A) \)
	contains no nontrivial \( \X \)-subgroups.
\end{theorem}

All known proofs of Theorem~\ref{WieHart} and its special cases
(see Section~\ref{Remarks}) use the Schreier conjecture asserting solvability
of the outer automorphism group \(\operatorname{Out}(S) \) of every simple group \( S \),
i.e.\ of the quotient of the automorphism group \( \Aut(S) \) of \( S \) by the group \( \Inn(S) \)
of inner automorphisms. Recall that the validity of the Schreier conjecture follows
from CFSG.

Theorem~\ref{WieHart} prompted Wielandt to introduce a new concept:
submaximal \( \X \)-subgroups.

\begin{definition}\label{StrongSubmax}
	\normalfont
	\cite[Definition~15.1]{Wie4} Let \( \X \) be a complete class
	of finite groups. A subgroup \( H \) of a finite
	group \( G \) is called a \textit{(strongly) submaximal
	\( \X \)-subgroup} or
	\textit{\( \X \)-submaximal in the sense of}~\cite{Wie4}
	(we write \( H \in \sm^\circ_\X(G) \))
	if there exists an embedding
	\[ \phi : G \hookrightarrow G^* \]
	of a group \( G \) in some finite group \( G^* \) such that
	\[ G^\phi \unlhd G^* \text{ and }
	H^\phi = X \cap G^\phi \text{ for some } X \in \m_\X(G^*). \]
\end{definition}

Less formally, \( H \in \sm^\circ_\X(G) \) if there is a group
\( G^* \) and its subgroup \( X \in \m_\X(G^*) \) such that
\( G \unlhd G^* \) and \( H = G \cap X \).

Since we can take \( G^* = G \), it is clear that
\( \m_\X(G) \subseteq \sm^\circ_\X(G) \).

Theorem~\ref{WieHart} can now be reformulated in this new language as follows.

\begin{theorem}[Wielandt-Hartley's theorem for strongly submaximal
	\( \X \)-subgroups]\label{WieHartStrongSubmax}
	Let \( G \) be a finite group and let \( \X \) be a complete class.
	If \( H \in \sm^\circ_\X(G) \), then \( N_G(H)/H \) contains no nontrivial \( \X \)-subgroups.
\end{theorem}

Fifteen years later, at the Santa Cruz conference \cite{Wie3},
Wielandt suggested a program for studying maximal \( \X \)-subgroups by projecting them into the
factors of a composition series. For that purpose, he came to a different, though close to original, definition of \( \X \)-submaximality.
He expected to find a generalization of a maximal \( \X \)-subgroup that would ``preserve as many
properties'' of Sylow \( p \)-subgroups and Hall \( \pi \)-subgroups
as possible, ``for example, compatibility with normal subgroups'' \cite[p.~170]{Wie3}.

\begin{definition}\label{Submax}
	\normalfont
	\cite[p.~170]{Wie3}
	Let \( \X \) be a complete class of finite groups. A subgroup
	\( H \) of \( G \) is called a
	\textit{submaximal \( \X \)-subgroup} or
	\textit{\( \X \)-submaximal in the sense of}~\cite{Wie3}
	(we write \( H \in \sm_\X(G) \))
	if there exists an embedding
	\[ \phi : G \hookrightarrow G^* \]
	of a group \( G \) in some finite group \( G^* \) such that
	\[ G^\phi {\unlhd\unlhd} G^* \text{ and }
	H^\phi = X \cap G^\phi \text{ for some } X \in \m_\X(G^*). \]
\end{definition}

Comparing Definitions~\ref{StrongSubmax} and~\ref{Submax}, one can see
that the only difference lies in the requirements on the embedding of \( G \)
into \( G^* \): in the first case \( G \) embeds as a normal subgroup,
while in the second case it embeds as a subnormal subgroup.
Since a normal subgroup is also subnormal,
\begin{equation}\label{inclusion}
	\m_\X(G) \subseteq \sm^\circ_\X(G) \subseteq \sm_\X(G).
\end{equation}

Definition~\ref{Submax} is \textit{a priori}
more general and intuitively more complicated than Definition~\ref{StrongSubmax}. But it fits
the goal of Wielandt's program better, because \( \X \)-submaximal (in the sense of \cite{Wie3}) subgroups
have the obvious inductive property
resembling properties of Sylow subgroups:
\begin{equation}\label{SubmaxProperty}
	\text{if } H \in \sm_\X(G) \text{ and }
	N{\unlhd\unlhd}G, \text{ then }
	H \cap N \in \sm_\X(N).
\end{equation}
Hence Definition~\ref{Submax} exactly satisfies the requirements that Wielandt
posed on the ``proper'' (``richtig'' \cite[p.~170]{Wie3}) generalization of maximal \( \X \)-subgroups.

If Definition~\ref{StrongSubmax} was equivalent to Definition~\ref{Submax}, it would also be ``richtig.''
However, it is not the case. In Section~\ref{Examples}, we provide a series of almost simple groups
\( G \) with socles isomorphic to the orthogonal groups \( \operatorname{P\Omega}_{4n}^+(q) \) such that
\( \sm^\circ_\X(G) \ne \sm_\X(G) \) for suitable classes~\( \X \).
Since Definitions~\ref{StrongSubmax} and~\ref{Submax} are not equivalent,
in what follows we refer to \( \X \)-subgroups from Definition~\ref{StrongSubmax}
as \textit{strongly} submaximal.

Now, it is natural to ask if submaximal \( \X \)-subgroups inherit main properties of strongly
submaximal \( \X \)-subgroups. In~\cite[5.4(a)]{Wie3}, Wielandt announced the following theorem.

\begin{theorem}[Wielandt-Hartley's theorem for submaximal \( \X \)-subgroups]\label{WieHartSubmax}
	Let \( G \) be a finite group and let \( \X \) be a complete class.
	If \( H \in \sm_\X(G) \), then \( N_G(H)/H \)
	contains no nontrivial \( \X \)-subgroups.
\end{theorem}

As in the case of Theorems~\ref{WieHart} and~\ref{WieHartStrongSubmax}, this result can be reformulated without using the notion
of a submaximal \( \X \)-subgroup.

\begin{theorem}[Wielandt-Hartley's theorem for subnormal subgroups]\label{WieHartSubnorm}
	Let \( G \) be a finite group and let \( \X \) be a complete class.
	If \( A \) is a subnormal subgroup of \( G \), then for
	every \( H \in \m_\X(G) \) the quotient \( N_A(H \cap A)/(H \cap A) \)
	contains no nontrivial \( \X \)-subgroups.
\end{theorem}

As is easily seen, the only difference of the latter assertion from Theorem~\ref{WieHart} is that
\( A \) is a subnormal (not necessarily normal) subgroup of \( G \).

In group theory properties of subnormal subgroups can be often extracted
from the corresponding properties of normal subgroups by means of straightforward
induction. However, some statements are indeed harder to prove for subnormal subgroups. That is the case for
the classical theorem of Wielandt \cite[Statements~7 and 9]{Wie0}, \cite[Ch. 2, (3.23)]{Suz1}
which asserts that in a finite group a subgroup generated by subnormal
subgroups is also subnormal. The same difficulty arises in the case of Theorem~\ref{WieHartSubnorm}.

As far as we know, a proof of Theorem~\ref{WieHartSubnorm} (and Theorem~\ref{WieHartSubmax}) never appeared.
In Section~\ref{Proof}, we fill a gap by proving this theorem.

It is worth mentioning that L.~A.~Shemetkov \cite[Theorem~7]{Shem} proved an important special
case of Theorem~\ref{WieHartSubnorm}. Namely, he showed that \textit{if
\( H \) is a maximal \( \pi \)-subgroup of a finite group \( G \), and \( A \)
is a subnormal subgroup of \( G \) which is not a \( \pi' \)-group, then \( H \cap A \neq 1 \)}.
In \cite[Proposition~8]{GR1}, W.~Guo and D.~Revin generalized this result to
an arbitrary complete class \( \X \). We use the latter in the proof of Theorem~\ref{WieHartSubnorm}.

We establish several applications of Wielandt-Hartley's theorem in this strong version.
To begin with, Theorem~\ref{WieHartSubmax} and inductive property~(\ref{SubmaxProperty})
allow us to prove the sufficient condition for conjugacy of submaximal \( \X \)-subgroups,
which, as Theorem~\ref{WieHartSubmax}, was announced in~\cite{Wie3}. 
Recall that if a group \( G \) with a subnormal series~(\ref{series}) contains a subgroup \( H \), then the projection of \( H \)
on \( G^i = G_{i-1}/G_i \), \(i=1,\dots,n\), is denoted by~\( H^i \).

\begin{corollary}\label{SubmaxProjections}
	Suppose that a group \( G \) possesses a subnormal
	series~{\rm(\ref{series})} and \( H, K \in \sm_\X(G) \) satisfy
	\[ H^i = K^i \text{ for all } i = 1, \dots, n. \]
	Then \( H \) and \( K \) are conjugate in the subgroup
	\( \langle H, K \rangle \).
\end{corollary}

Since \( \m_\X(G) \subseteq \sm_\X(G) \), the same assertion holds for
\( \X \)-maximal subgroups.

In~\cite[pp.~30--31]{GR1}, it was noticed that Theorem~\ref{WieHartSubmax} implies a
characterization of submaximal \( \X \)-subgroups in direct products:

\begin{corollary}\label{DirectProd}
	Let \( G = G_1 \times \dots \times G_n \) be a direct product of its subgroups \( G_i \), \(i=1,\dots,n\). Then for every complete
	class \( \X \),
	\[ \sm_\X(G) = \{ \langle H_1, \dots, H_n \rangle \mid H_i \in
	\sm_\X(G_i), i = 1, \dots, n \}. \]
\end{corollary}

This corollary, helpful in inductive arguments, is an analogue of the well-known property of maximal
\( \X \)-subgroups \cite[Proposition~10]{GR1}.

In his talk in 1979, Wielandt posed a problem of reversing Theorem~\ref{WieHartSubnorm} for
classes of \( \pi \)-groups \cite[p.~171, Problem~(i)]{Wie3}: Must a subgroup \( A \) be subnormal
in \( G \) if the order of \( N_A(H \cap A)/(H \cap A) \) is not divisible by
any number in \( \pi \) for all sets of primes \( \pi \) and all maximal
\( \pi \)-subgroups \( H \) of~\( G \)? In 1991, P.~Kleidman obtained a positive
answer with the help of CFSG \cite{Kleid}. Combining Theorem~\ref{WieHartSubnorm} with Kleidman's result,
we come to the following criterion of subnormality.
\begin{corollary}\label{Subnormality}
	A subgroup \( A \) of a group \( G \) is subnormal if and only if for every
	complete class \( \X \) and every maximal \( \X \)-subgroup \( H \)
	in \( G \), the quotient \( N_A(H \cap A)/(H \cap A) \) contains no nontrivial \( \X \)-subgroups.
\end{corollary}

To sum up, this article contributes to Wielandt's program of
studying maximal \( \X \)-subgroups, which Wielandt himself deemed to be a development
of the H\"older program. We would also like to mention some recent progress in this direction
made in \cite{GR1}, \cite{GR3}, and \cite{GRsurv}, where, in particular, the authors
suggested an inductive algorithm of finding maximal \( \X \)-subgroups in a finite group
provided all submaximal \( \X \)-subgroups in finite simple groups are known for a given class~\( \X \).

\section{Examples of submaximal but not strongly submaximal \( \X \)-subgroups}\label{Examples}

As mentioned in the introduction, we will find desired examples among
almost simple groups. Recall that a group \( G \) is \textit{almost simple} with socle
\( S \), if
\[ S \le G \le \Aut(S), \]
where a nonabelian simple group \( S \) is identified with the group \( \Inn(S) \) of its inner automorphisms.

The following lemma refines the definition of strong submaximality
for almost simple groups.

\begin{lemma}\label{AlmostSimpleSrongSubmax}
	Let \( G \) be an almost simple group with socle \( S \notin \X \).
	Then \( H \in \sm^\circ_\X(G) \) if and only if there exists a group
	\( G^* \) such that \( G \unlhd G^* \le \Aut(S) \) and
	\( H = K \cap G \) for some \( K \in \m_\X(G^*) \).
\end{lemma}

\begin{proof}
	The ``if'' part follows from the definition of a strongly submaximal
	\( \X \)-subgroup.

	To show the converse, take a group \( G^* \) of the smallest order such that
	\( G \unlhd G^* \) and \( H = K \cap G \) for some \( K \in \m_\X(G^*) \).
	It is clear that \( G^* = GK \) and, in particular, \( G^* / G \in \X \).

	We claim that \( G^* \) does not contain any nontrivial normal
	\( \X \)-subgroups. Indeed, suppose that \( 1 \ne U \in \X \) is a
	normal subgroup in \( G^* \).
	If \( G \cap U \ne 1 \), then \( G \cap U \) being a normal subgroup of \( G \)
	contains the unique minimal normal subgroup \( S \) of \( G \),
	contrary to the assumption that \( S \notin \X \). Therefore, \( G \cap U = 1 \), so
	\( G \simeq \overline{G} \unlhd \overline{G^*} \), where
	\( \overline{\phantom{x}}: G^* \rightarrow G^*/U \)
	denotes the canonical epimorphism.
	Moreover, \( K \in \m_\X(G^*) \) implies \( U \le K \), and hence
	\( \overline{K} \in \m_\X\left(\overline{G^*}\right) \). Dedekind's
	lemma (see, e.g.,~\cite[Ch.~1, Theorem~3.14]{Suz1}) yields
	\[ HU = (G \cap K)U = GU \cap K. \]
	It follows that \( \overline{H} = \overline{G} \cap \overline{K} \),
	so we arrive at a contradiction with the minimality of \( G^* \).

	Since \( S \) is a characteristic subgroup of \( G \), both
	\( S \) and its centralizer \( C_{G^*}(S) \) are normal in \( G^* \).
	The group \( G \) is almost simple, so
	\( C_{G^*}(S) \cap G = C_G(S) = 1 \). Hence \( C_{G^*}(S) \)
	can be isomorphically embedded into an \( \X \)-group  \( G^*/G \),
	so it is an \( \X \)-group itself. As \( G^* \) contains no
	nontrivial normal \( \X \)-subgroups, \( C_{G^*}(S) = 1 \).
	Thus, \( G^* \) is isomorphic to a subgroup of \( \Aut(S) \).
\end{proof}

For the rest of this section, we fix the following notation.

Let \( S \) be a simple group \( D_{2n}(q) \simeq \operatorname{P\Omega}_{4n}^+(q) \),
where \( q \) is an odd prime and \( {n > 2} \). Set \( A = \Aut(S) \).
Denote by \( \widehat{S} \) the group of inner-diagonal automorphisms of \( S \),
see~\cite[7.1, 8.4.7, 12.2]{Car} or \cite[Definitions~2.5.10 and 1.15]{GLS}.
Then \( S \le \widehat{S} \le \Aut(S) \).

\begin{figure*}
	\includegraphics[width=119mm]{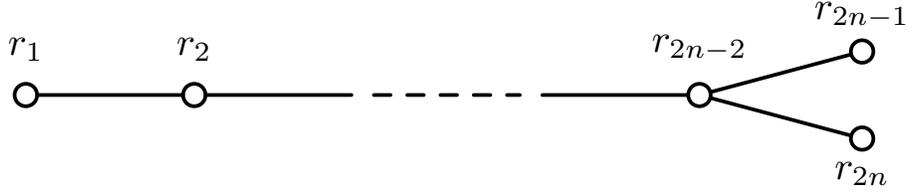}
\caption{$D_{2n}$ root system}
\label{p91}  
\end{figure*}

Let \( \Pi = \{ r_1, \dots, r_{2n} \} \) be a fundamental root system of
type \( D_{2n} \). Its Dynkin diagram is indicated in Fig.~\ref{p91}.
The symmetry of the graph in Fig.~\ref{p91} corresponding to the
transposition of roots \( r_{2n-1} \) and \( r_{2n} \) induces a so-called
graph automorphism \( \gamma \in A \) of order \( 2 \),
see~\cite[Proposition~12.2.3]{Car}. Since \( q \) is a prime,
\( A = \langle \widehat{S}, \gamma \rangle \)
(see~\cite[Theorem~12.5.1]{Car} or \cite[Theorem~2.5.12]{GLS}).
It is known that \( A/S \) is isomorphic to the dihedral group of order
\( 8 \) \cite[Theorem~2.5.12(j)]{GLS} and contains the normal subgroup
\( \widehat{S}/S \) which is isomorphic to the elementary abelian group of
order \( 4 \) (see~\cite[8.4.7 and 8.6]{Car} or
\cite[Theorem~2.5.12(c)]{GLS}).

Let \( P_0 \), \( P_1 \) and \( P_2 \) denote parabolic subgroups of
\( S \) containing the same Borel subgroup and corresponding to the following sets of roots
(see~\cite[8.2.2 and 8.3]{Car}):
\( \Pi \setminus \{ r_{2n-1}, r_{2n} \} \),
\( \Pi \setminus \{ r_{2n-1} \} \) and \( \Pi \setminus \{ r_{2n} \} \).
The next lemma states some properties of these subgroups.

\begin{lemma}\label{PropertiesOfP}
	In the above notation, the following hold.
	\begin{enumerate}[{\normalfont (i)}]
		\item \( N_S(P_i) = P_i \) for every \(i = 0, 1, 2\).
		\item \( P_i \) is a maximal subgroup of \( S \)
			for \( i = 1, 2 \), and if \( P_0 < P < S \),
			then \( P \in \{ P_1, P_2 \} \).
		\item \( P_1^\gamma = P_2 \) and
			\( P_2^\gamma = P_1 \).
		\item \( P_0^\gamma = P_0 \).
		\item There exists an abelian subgroup
			\( \widehat{T} \) of \( \widehat{S} \) such that
			\( \widehat{S} = \widehat{T}S \) and
			\( \widehat{T} \) normalizes \( P_i \) for
			\( i = 0, 1, 2 \). In particular,
			\( N_{\widehat{S}}(P_i) = \widehat{T}P_i \).
		\item Subgroups \( P_1 \) and \( P_2 \) are not
			conjugate in \( S \).
	\end{enumerate}
\end{lemma}
\begin{proof}
	See~\cite[Theorem~8.3.3]{Car} for~(i) and
	\cite[Theorems~8.3.2 and~8.3.4]{Car} for~(ii). Statements (iii)
	and (iv) follow from the definitions of subgroups \( P_i \)
	and automorphism \( \gamma \), see~\cite[Proposition~12.2.3]{Car}.
	It follows from \cite[Proposition~2.6.9]{GLS} that the Cartan subgroup
	\( \widehat{T} \) is an abelian subgroup of \( \widehat{S} \) such that \( \widehat{S} = \widehat{T}S \)
	and \( \widehat{T} \) normalizes \( P_i \) for \( i = 0, 1, 2 \).
	Now, if \( g \in N_{\widehat{S}}(P_i) \) then \( g = xy \) for
	\( x \in \widehat{T} \) and \( y \in S \). Hence
	\[ P_i = P_i^g = P_i^{xy} = P_i^y, \]
	so \( y \in N_S(P_i) = P_i \) by~(i). Therefore,
	\( g \in \widehat{T}P_i \) and
	\( N_{\widehat{S}}(P_i) \le \widehat{T}P_i \).
	The reverse inclusion is clear, so~(v) is proved.
	Statement~(vi) follows from \cite[Theorem~8.3.3]{Car} or \cite[Theorem~2.6.5(c)]{GLS}.
\end{proof}

Let \( G \) denote a subgroup of \( \widehat{S} \) containing \( S \)
such that  \( G/S \) has order \( 2 \) and it is not normal in the dihedral group \( A/S \)
(there are exactly two such subgroups; they are normal in \( \widehat{S} \)
and are permuted by \( \gamma \), see~\cite[Theorem~2.5.12(j)]{GLS}).
Observe that \( G \) is subnormal in \( A \).

It is easily seen that the class \( \X \) of all finite groups with nonabelian
composition factors having order less than \( |S| \) is complete.
Since \( A/S \) is solvable, the set \( \m_\X(G) \) coincides with
the set of subgroups of \( G \) maximal among those that do not contain \( S \).
The following proposition shows that there are submaximal \( \X \)-subgroups
which are not strongly submaximal.

\begin{proposition}\label{Counterexample}
	If \( H = N_G(P_0) \), then
	\( H \in \sm_\X^{\phantom{\circ}}(G) \setminus  \sm^{\circ}_\X(G) \).
\end{proposition}
\begin{proof}
	First, we prove that \( H \in \sm_\X(G) \). It suffices to show that
	\( N_A(P_0) \) is maximal in \( A \), because in this case \( N_A(P_0) \)
	is \( \X \)-maximal in \( A \) and
	\[ H = N_G(P_0) = G \cap N_A(P_0) \in \sm_\X(G). \]
	
	If \( N_A(P_0) \) is not maximal in \( A \), then there is
	a subgroup \( M \) with \( N_A(P_0) < M < A \).
	By Statements~(iv) and (v) of Lemma~\ref{PropertiesOfP}, we have
	\( A = \langle \widehat{S}, \gamma \rangle \le N_A(P_0)S \).
	Therefore,
	\[ N_A(P_0)/(N_A(P_0) \cap S) \simeq N_A(P_0)S/S = A/S = MS/S
	\simeq M/(M \cap S). \]
	Hence, the equality \( N_A(P_0) \cap S = M \cap S \) does not hold.
        By Lemma~\ref{PropertiesOfP}(i),
	\[ P_0 = N_S(P_0) = N_A(P_0) \cap S < M \cap S. \]
	As \( M \cap S < S \), Lemma~\ref{PropertiesOfP}(ii)
	yields \( M \cap S = P_i \) for \( i = 1 \) or \( i = 2 \).
	Since \( M \cap S \unlhd M \), we obtain
	\[ M \le N_A(P_i) \text{ and } A = MS = N_A(P_i)S. \]
	This implies that the conjugacy class of \( P_i \) is fixed by \( A \),
	contrary to the fact that \( \gamma \in A \) permutes the conjugacy classes of
	\( P_1 \) and \( P_2 \) in \( S \)
	(see Statements~(iii) and (vi) of Lemma~\ref{PropertiesOfP}).
	Thus, \( N_A(P_0) \) is maximal in~\( A \), and we are done.

	Let us show that \( H \notin \sm^{\circ}_\X(G) \). Suppose the
	contrary. Then by Lemma~\ref{AlmostSimpleSrongSubmax}, there is
	a subgroup \( G^* \) of \( A \) such that \( G \unlhd G^* \) and
	\( H = K \cap G \) for some \( K \in \m_\X(G^*) \), i.e.\ for some
	maximal subgroup \( K \) of \( G^* \) that does not contain~\( S \).
	Since \( G^* \) normalizes \( G \), a subgroup \( G^*/S \) of the
	dihedral group \( A/S \) lies in the normalizer \( N_{A/S}(G/S) \),
	which is equal to \( \widehat{S}/S \).
	Therefore, \( G^* \le \widehat{S} \).

	We claim that \( K = N_{G^*}(P_0) \). As \( K \) is maximal in
	\( G^* \), it suffices to show that \( K \) normalizes \( P_0 \).
	Indeed, \( K \) normalizes \( K \cap S \) and, by Lemma~\ref{PropertiesOfP}(i),
	\[ K \cap S = K \cap G \cap S = H \cap S = N_G(P_0) \cap S =
	   N_S(P_0) = P_0. \]
	If \( \widehat{T} \) is a subgroup from
	Lemma~\ref{PropertiesOfP}(v), then for all \( i = 0, 1, 2, \)
	\[ N_{\widehat{S}}(P_i) = \widehat{T}P_i. \]
	By Dedekind's lemma,
	\[ N_{G^*}(P_i) = G^* \cap N_{\widehat{S}}(P_i) =
	   G^* \cap \widehat{T}P_i = (G^*\cap \widehat{T})P_i. \]
	Since
	\[ K = N_{G^*}(P_0) = (G^* \cap \widehat{T})P_0 \le
	   (G^* \cap \widehat{T})P_1 = N_{G^*}(P_1), \]
	the maximality of \( K \) implies \( N_{G^*}(P_0) = N_{G^*}(P_1)\).
	Lemma~\ref{PropertiesOfP}(ii) yields
	\[ P_0 = N_S(P_0) = S \cap N_{G^*}(P_0) = S \cap N_{G^*}(P_1) =
	   N_S(P_1) = P_1, \]
	contrary to the fact that \( P_0 \neq P_1 \).
\end{proof}

\section{The proof of Wielandt-Hartley's theorem for submaximal \( \X \)-subgroups}\label{Proof}

Given a group \( G \), the set of prime divisors of the order of \( G \) is denoted by~\( \pi(G) \).
If \( \pi \) is an arbitrary set of primes, then \( G \) is called a \textit{\( \pi \)-group}
provided \( \pi(G) \subseteq \pi \), and a \textit{\( \pi' \)-group} whenever \( \pi(G) \cap \pi = \varnothing \).

Given a class \( \X \), set \( \pi(\X) = \bigcup_{X \in \mathfrak{X}} \pi(X) \).
Denote by \( O_\X(G) \) the largest normal \( \X \)-subgroup of \( G \).
If \( \X \) is a class of \( p \)-groups, \( \pi \)-groups or \( \pi' \)-groups, then
we denote the subgroup \( O_\X(G) \)  by \( O_p(G) \), \( O_\pi(G) \) or \( O_{\pi'}(G) \) respectively.

Observe that for every \( p\in\pi(\X) \) the group of order \( p \) lies in \( \X \). Therefore, 
a group contains no nontrivial \( \X \)-subgroups if and only if it is a \( \pi' \)-group for \( \pi=\pi(\X) \). 
Thus, Theorem~\ref{WieHartSubnorm} and hence Theorem~\ref{WieHartSubmax} are equivalent to the following proposition.

\begin{proposition}\label{refined}
	Let \( \X \) be a complete class of finite groups and \( \pi = \pi(\X) \).
	Let \( G \) be a finite group and let \( H \in \m_\X(G) \).
	If \( A \) is subnormal in \( G \), then \( N_A(H \cap A)/(H \cap A) \) is a \( \pi' \)-group.
\end{proposition}
\begin{proof}
	Suppose the contrary, and let \( G \) be a counterexample of minimal order.
	If \( A = 1 \) or \( A = G \), then we immediately get a contradiction,
	so \( 1 < A < G \) and, in particular, \( G \) is not a simple group.

	\begin{lemma}\label{triv}
		\( O_\X(G) = 1 \) and \( O_{\pi'}(G) = 1.\) In particular, \( G \) contains no abelian minimal subnormal subgroups.
	\end{lemma}
	\begin{proof}
		Suppose that \( K = O_\X(G) > 1 \). Let
		\( \overline{\phantom{A}} : G \to G/K \) denote the canonical epimorphism.
		Clearly, \( K \leq H \), so \( \overline{H} \) is a maximal \( \X \)-subgroup of \( \overline{G} \).
		Therefore,
		\[ N_{\overline{AK}}(\overline{AK} \cap \overline{H})/(\overline{AK} \cap \overline{H}) \simeq N_{AK}(AK \cap H)/(AK \cap H), \]
		where by the minimality of \( G \), the quotient on the left-hand side is a \( \pi' \)-group.
		In virtue of Dedekind's lemma,
		\[ N_{AK}(AK \cap H)/(AK \cap H) = N_{AK}((A \cap H)K)/((A \cap H)K). \]
		Since \( K \) is normal in \( G \), we have \( KN_A(A \cap H) \leq N_{AK}((A \cap H)K) \).
		Thus,
		\[ N_{AK}((A \cap H)K)/((A \cap H)K) \geq KN_A(A \cap H)/((A \cap H)K), \]
		and the group on the right-hand side is also a \( \pi' \)-group. Its order is equal to
		\begin{equation*}
			\begin{split}
				|KN_A(A \cap H) : K(A \cap H)| = \frac{|K||N_A(A \cap H)||A \cap H \cap K|}{|K \cap N_A(A \cap H)||K||A \cap H|} =\\
				= |N_A(A \cap H) : A \cap H||A \cap K : N_{A \cap K}(A \cap H)|.
			\end{split}
		\end{equation*}
		Since the second factor is an integer, it follows that \( N_A(A \cap H)/(A \cap H) \) is a \( \pi' \)-group,
		contrary to the choice of \( G \). Therefore, \( O_\X(G) = 1 \).

		Suppose that \( O_{\pi'}(G) > 1 \), and denote by \( K \) a minimal normal subgroup of \( G \)
		contained in \( O_{\pi'}(G) \). Then \( K > 1 \) and \( K \) normalizes \( A \) by \cite[Theorem~2.6]{isaacs}.
		Let \( P \) be a Sylow \( p \)-subgroup of \( AK \) for some \( p \in \pi \). The subgroup \( A \) is normal in \( AK \)
		and \( K \) is a \( \pi' \)-group, so \( P \leq A \). Since this holds for every \( p \in \pi \), we derive that
		all Sylow \( p \)-subgroups and hence all \( \pi \)-subgroups of \( AK \) lie in \( A \). Therefore, \( AK \cap H = A \cap H \).
		As \( K \) is normal in	\( G \), we have
		\[ N_A(A \cap H) \leq N_{AK}((A \cap H)K) = N_{AK}((AK \cap H)K) = N_{AK}(AK \cap HK), \]
		where the last equality holds by Dedekind's lemma.

		Let \( \overline{\phantom{A}} : G \to G/K \) be the canonical epimorphism.
		In view of \cite[Proposition~4]{GR1}, \( \overline{H} \) is a maximal \( \X \)-subgroup of \( \overline{G} \).
		By the minimality of \( G \), the quotient \( N_{\overline{A}}(\overline{A} \cap \overline{H})/(\overline{A} \cap \overline{H}) \)
		is a \( \pi' \)-group. Since \( KN_A(A \cap H)/(K(A \cap H)) \) is a subgroup of
		\[ N_{AK}(AK \cap HK)/(AK \cap HK)
		\simeq N_{\overline{A}}(\overline{A} \cap \overline{H})/(\overline{A} \cap \overline{H}), \]
		\( KN_A(A \cap H)/K(A \cap H) \) is a \( \pi' \)-group.
		As \( K \cap H = 1 \),
		\[ |KN_A(A \cap H) : K(A \cap H)||K \cap N_A(A \cap H)| = |N_A(A \cap H) : A \cap H|, \]
		and hence \( N_A(A \cap H)/(A \cap H) \) is also a \( \pi' \)-group, a contradiction.
		Thus, \( O_{\pi'}(G) = 1 \).

		If \( p \) is a prime, then \( O_p(G) \) lies either in \( O_{\pi'}(G) \) or \( O_\X(G) \).
		It follows that \( O_p(G) = 1 \) for every prime \( p \). As a consequence,
		all minimal subnormal subgroups of \( G \) are nonabelian.
	\end{proof}

	\begin{lemma}\label{centr}
		If \( S \) is a minimal subnormal subgroup of \( G \) and \( S\nsubseteq A \), then \( {[S, A] = 1} \).
	\end{lemma}
	\begin{proof}
		It follows from \cite[Lemma~9.17]{isaacs} that the normal closure  \( M = S^G \)
		of \( S \) in \( G \) is a minimal normal subgroup of~\( G \).
		By \cite[Theorem~2.6]{isaacs}, \( M \) normalizes~\( A \), so \( A \) is normal in \( N = \langle M, A \rangle \).
		The normal closure \( S^N \) of \( S \) in \( N \)  is a minimal normal
		subgroup of \( N \) and does not lie in \( A \). Hence \( S^N \cap A = 1 \) and
		\( [S, A] \leq [S^N, A] = 1 \), as required.
	\end{proof}

	\begin{lemma}\label{factor}
		Let \( M = S_1 \times \dots \times S_n \) be a subnormal subgroup of \( G \),
		where \( S_i \), \( i = 1, \dots, n \), are simple groups. Then
		\[ M \cap H = (S_1 \cap H) \times \dots \times (S_n \cap H). \]
	\end{lemma}
	\begin{proof}
		Note that \( M \cap H \geq S_i \cap H \) for all \( i = 1, \dots, n \), so
		\[ M \cap H \geq (S_1 \cap H) \dots (S_n \cap H). \]
		By \cite[Proposition~10]{GR1}, there are \( L_i \in \sm_\X(S_i) \), \( i = 1, \dots, n \),
		such that \( M \cap H = L_1 \times \dots \times L_n \). Clearly \( L_i \leq S_i \cap H \) for \( i = 1, \dots, n \).
		It follows that
		\[ M \cap H \leq (S_1 \cap H) \dots (S_n \cap H), \]
		and we are done.
	\end{proof}

	We proceed with the proof of Proposition~\ref{refined}. Set \( X = \langle N_A(A \cap H), \, H \rangle \)
	and \( L = A \cap X \). Observe that \( H \) is a maximal \( \X \)-subgroup of \( X \),
	\( L \) is subnormal in \( X \), and \( L \cap H = A \cap H \). Furthermore,
	\[ N_L(L \cap H) = N_{A \cap X}(A \cap H) = N_A(A \cap H). \]
	If \( X < G \), then by the minimality of \( G \) the quotient \( N_L(L \cap H)/(L \cap H) \) is a \( \pi' \)-group.
	It follows that \( N_A(A \cap H)/(A \cap H) \) is also a \( \pi' \)-group, a contradiction.
	Hence we may assume that \( G = \langle N_A(A \cap H), \, H \rangle \).

	Let \( M \) be a minimal normal subgroup of \( G \). Then \( M = S_1 \times \dots \times S_n \),
	where \( S_i \), \( i = 1, \dots, n \), are nonabelian simple groups.
	Since \( A \cap M \) is subnormal in \( M \), it is a (possibly empty) product of some \( S_i \), \( i = 1, \dots, n \).
	Without loss of generality, we may assume that \( A \cap M = S_1 \times \dots \times S_k \), \( 0 \leq k \leq n \).
	Applying Lemma~\ref{factor} to \( M \) and \( A \cap M \), we obtain
	\[ M \cap H = (S_1 \cap H) \times \dots \times (S_n \cap H ) \text{ and }
	A \cap M \cap H = (S_1 \cap H) \times \dots \times (S_k \cap H ). \]
	Since \( M = (A \cap M) \times S_{k+1} \times \dots \times S_n \), it follows that
	\[ M \cap H = (A \cap M \cap H)\times(S_{k+1} \cap H)\times \dots \times(S_n \cap H). \]
	Clearly, \( N_A(A \cap H) \) normalizes \( A \cap M \cap H \). For every \( i \in \{ k+1, \dots, n \} \),
	the subgroup \( A \) centralizes \( S_i \) due to Lemma~\ref{centr}, so \( N_A(A \cap H) \) centralizes~\( {S_i \cap H} \).
	Consequently, \( N_A(A \cap H) \) normalizes all factors constituting \( M \cap H \), and therefore
	it normalizes \( M \cap H \) itself. Obviously, \( H \) also normalizes \( M \cap H \).
	Since \( G = \langle N_A(A \cap H), H \rangle \), we derive that \( M \cap H \) is normal in \( G \).

	Now, \( M \cap H \) is a normal \( \X \)-subgroup of \( G \), so \( M \cap H = 1 \) by Lemma~\ref{triv}.
	It follows from \cite[Proposition~8]{GR1} that \( M \) is a \( \pi' \)-group, hence \( M = 1 \) again by Lemma~\ref{triv}.
	Thus, \( G \) must be a simple group which gives us a final contradiction.
\end{proof}

\section{Applications}\label{Corollaries}

In this section we prove Corollaries~\ref{SubmaxProjections}--\ref{Subnormality} from Introduction
and thus show how one can apply the strong version of Wielandt-Hartley's theorem. 

Given a complete class \( \X \), a group \( G \) is said to be \textit{\( \X \)-separable}, if \( G \) possesses a subnormal series~(\ref{series})
such that each factor \( G^i \) either belongs to \( \X \) or contains no nontrivial \( \X \)-subgroups. Clearly, a subgroup of an
\( \X \)-separable group is also \( \X \)-separable.

\begin{lemma}\label{Chunikhin}
	{\rm \cite[12.10]{Wie4}}
	Let \( G \) be an \( \X \)-separable group. Then all maximal
	\( \X \)-subgroups of \( G \) are conjugate.
\end{lemma}

We note that Lemma~\ref{Chunikhin} is equivalent to Chunikhin's lemma
on \( \pi \)-separable groups \cite[Ch.~5, Theorem~3.7]{Suz}.

\smallskip

\textbf{Proof of Corollary~\ref{SubmaxProjections}.}
	We use induction on the length \( n \) of a
	series~(\ref{series}). Base \( n = 0 \) is clear. 

	Property~(\ref{SubmaxProperty}) yields \( H \cap G_1 \in \sm_\X(G_1) \)
	and \( K \cap G_1 \in \sm_\X(G_1) \). By induction hypothesis,
	subgroups \( H \cap G_1 \) and \( K \cap G_1 \) are conjugate in
	\( \langle H \cap G_1, K \cap G_1 \rangle \le
	   \langle H, K \rangle \cap G_1 \). Without loss of generality,
	we may assume that \( H \cap G_1 = K \cap G_1 \). Equality
	\( H^1 = K^1 \) implies \( HG_1 = KG_1 \). Set
	\[ G^* = HG_1 = KG_1 \quad \text{ and } \quad
	   T = H \cap G_1 = K \cap G_1. \]
	Then \( H, K \le N_{G^*}(T) \) and \( N_{G^*}(T) = HN_{G_1}(T) \).
	Moreover,
	\[ N_{G_1}(T) = G_1 \cap N_{G^*}(T) \unlhd N_{G^*}(T) =
	   HN_{G_1}(T). \]
	Therefore, \( N_{G^*}(T)/N_{G_1}(T) \) is
	isomorphic to a quotient group of \( H \), so it lies in \( \X \).
	Next, \( T = H \cap G_1 \in \sm_\X(G_1) \), and
	Theorem~\ref{WieHartSubmax} yields that the group \( N_{G_1}(T)/T \) 
	contains no nontrivial \( \X \)-subgroups. Now, \( N_{G^*}(T) \)
	is \( \X \)-separable because it possesses a (sub)normal series
	\[ N_{G^*}(T) \ge N_{G_1}(T) \ge T \ge 1, \]
	where every factor either lies in \( \X \) or contains no
	nontrivial \( \X \)-subgroups. Projections of \( H \)
	and \( K \) on factors of that series are maximal \( \X \)-subgroups
	in respective factors, hence \( H, K \in \m_\X(N_{G^*}(T)) \).
	A subgroup \( J = \langle H, K \rangle \) of \( N_{G^*}(T) \)
	is also \( \X \)-separable, and \( H, K \in \m_\X(J) \).
	By Lemma~\ref{Chunikhin}, \( H \) and \( K \) are conjugate
	in \( J \).

Inclusion~(\ref{inclusion}) and Corollary~\ref{SubmaxProjections} immediately imply

\begin{corollary}\label{MaxProjections}
	Suppose that the group \( G \) possesses a subnormal
	series~{\rm(\ref{series})} and \( H, K \in \m_\X(G) \) are such that
	\[ H^i = K^i \text{ for all } i = 1, \dots, n. \]
	Then \( H \) and \( K \) are conjugate in the subgroup
	\( \langle H, K \rangle \).
\end{corollary}
The classical Kaloujnine-Krasner theorem \cite{KK} provides a source of groups \( G \) possessing
subgroups \( H \) and \( K \) such that the projections \( H^i \) and \( K^i \) on the sections \( G^i \) of a subnormal  series~(\ref{series}) coincide, while \( H \) and \( K \) are not even isomorphic.
In contrast, Corollaries~\ref{SubmaxProjections} and~\ref{MaxProjections} mean that, up to conjugation, every submaximal
and, in particular, every  maximal \( \X \)-subgroup
is uniquely determined by its projections on the sections of a subnormal series.
In the case where all terms of a series~(\ref{series}) are normal in \( G \),
Corollary~\ref{MaxProjections} is proved in~\cite[Hauptsatz~14.1]{Wie4} and in~\cite[Ch.~5, (3.21) and (3.21)\('\)]{Suz}.
To realize the effectiveness of the notion of submaximal
\( \X \)-subgroups and the power of inductive
property~(\ref{SubmaxProperty}), one can compare the proof of
Corollary~\ref{SubmaxProjections} with the proof of a weaker statement
\cite[Ch.~5, 3.21]{Suz}. Property~(\ref{SubmaxProperty}) and Theorem~\ref{WieHartSubmax} allow us to
use the induction hypothesis straightaway, which makes the reasoning much
simpler and shorter.

\textbf{Proof of Corollary~\ref{DirectProd}.}
	First, let \( H\in \sm_\X(G) \). Property~(\ref{SubmaxProperty}) yields \( H\cap G_i\in \sm_\X(G_i) \)
	for all \( i = 1, \dots, n \). Let \( \rho_i:G\rightarrow G_i \) be the coordinate projection mapping and \( H_i=H^{\rho_i} \).
	Since \( H\cap G_i\unlhd H \) and \( \rho_i \) acts identically on \(G_i\),
	\[ H\cap G_i=(H\cap G_i)^{\rho_i}\unlhd H^{\rho_i}=H_i. \]
	Therefore, \( H_i\leq N_{G_i}(H\cap G_i) \). Theorem~\ref{WieHartSubmax} implies that
	\( N_{G_i}(H\cap G_i)/(H\cap G_i) \) contains no nontrivial \( \X \)-subgroups.
	Consequently, the image of \( H_i \) in \( N_{G_i}(H\cap G_i)/(H\cap G_i) \) is trivial and
	\( H_i=H\cap G_i \).  It follows that
	\[ \langle H\cap G_1,\dots, H\cap G_n\rangle\leq H\leq \langle H_1,\dots, H_n\rangle=\langle H\cap G_1,\dots, H\cap G_n\rangle, \]
	and \( H=\langle H\cap G_1,\dots, H\cap G_n\rangle \), as desired.

	To prove the converse, take arbitrary \( H_i\in\sm_\X(G_i) \) for all \( i=1,\dots,n \). We may assume that for every \( i \) a group
	\( G_i^* \) exists, in which \( G_i \) is subnormal and \( H_i=K_i\cap G_i \) for a suitable \( K_i\in \m_\X(G_i^*) \).
	It is easy to see that
	\[ K=\langle K_1,\dots, K_n\rangle= K_1\times\dots\times K_n\in \m_\X(G^*),\text{ where }G^*=G_1^*\times\dots\times G_n^*. \]
	Moreover, \( H_i=K_i\cap G_i\,{\unlhd\unlhd}\,K_i\unlhd K \), whence \( H_i\,{\unlhd\unlhd}\,K\cap G \) and
	\[ H_i=H_i^{\rho_i}\,{\unlhd\unlhd}\,(K\cap G)^{\rho_i},\]
	where again \( \rho_i:G^*\rightarrow G_i^* \) is the coordinate projection mapping.
	Theorem~\ref{WieHartSubmax} implies that \( N_{(K\cap G)^{\rho_i}}(H_i)=H_i \),
	which is possible only if \( H_i=(K\cap G)^{\rho_i} \). Thus,
	\[ \langle H_i \mid i = 1, \dots, n \rangle \leq K \cap G \leq \langle (K\cap G)^{\rho_i} \mid i = 1, \dots, n \rangle=
	   \langle H_i \mid i = 1, \dots, n \rangle. \]
	In view of \( K\in\m_\X(G^*) \) and \( G{\unlhd\unlhd} G^* \), we obtain
	\[ \langle H_i \mid i = 1, \dots, n \rangle = K \cap G \in \sm_\X(G), \]
	and this completes the proof.
	\qed

\smallskip

\textbf{Proof of Corollary~\ref{Subnormality}.} It suffices to establish the following
\begin{proposition}\label{Subnormality1}
	Let \( A \) be a subgroup of a finite group \( G \). Then the following statements are equivalent.
	\begin{enumerate}[{\normalfont (i)}]
		\item \( A\,{\unlhd\unlhd}\,G \).
		\item \( N_A(H\cap A)/(H\cap A) \) contains no nontrivial \( \X \)-subgroups
			for every complete class \( \X \) and every \( H\in\m_\X(G) \).
		\item \( H\cap A \) is a Sylow \( p \)-subgroup of \( A \)
			for all primes \( p \) and every Sylow \( p \)-subgroup \( H \) of~\( G \).
	\end{enumerate}
\end{proposition}
\begin{proof}
	(i)\( \Rightarrow \)(ii) is the statement of Theorem~\ref{WieHartSubnorm}.\\
	(ii)\( \Rightarrow \)(iii). Let \( p \) be a prime and let \( H \) be a Sylow \( p \)-subgroup of~\( G \).
	Applying~(ii) to the situation where \( \X \) is the class of all \( p \)-groups,
	we obtain that \( H\cap A \) is a Sylow \( p \)-subgroup of \( N_A(H\cap A) \).
	Now it follows from the well-known consequence of the Sylow theorem (see, e.g., \cite[Ch.~2, (2.5)]{Suz1})
	that \( H\cap A \) is a Sylow \( p \)-subgroup of~\( A \).\\
	(iii)\( \Rightarrow \)(i) is the main result of~\cite{Kleid}.
\end{proof}

\section{Concluding remarks}\label{Remarks}

Wielandt-Hartley's theorem for normal subgroups (Theorem~\ref{WieHart}) is an invaluable tool for studying subgroup
structure of finite groups. It can be found in Suzuki's classic book \cite[Ch.~5, (3.20) and (3.20)\( ' \)]{Suz}.
Although Wielandt proved this theorem in his lectures \cite{Wie4} delivered at
T\"ubingen in 1963--64, various particular cases of that result were
independently proved by different authors without mentioning Wielandt.
One of the reasons was that the lectures were first
published only in 1994, when the collection of Wielandt's mathematical works
appeared.

The first published proof (1971) of that theorem is by Hartley
\cite[Lemmas~2 and 3]{Hart}. It was obtained for the case when
\( \X \) is a class of \( \pi \)-groups. A similar result was proved
by Shemetkov~\cite{Shem1} in 1972. Both Hartley and Shemetkov used this
version of the theorem as a technical instrument for studying the
well-known \( D_\pi \)-problem: Is it true that the class of groups with
all maximal \( \pi \)-subgroups being conjugate is closed under extensions?
This problem was posed by Wielandt at the XIII International Congress of
Mathematicians in Edinburgh in 1958 \cite{Wie1} and traces back to
P.~Hall's theorem~\cite[Theorem~D5]{Hall}.

Another special case of Theorem~\ref{WieHart} can be obtained by fixing
a nonabelian simple group \( S \) and considering the complete class
\( \X \) of all finite groups with all composition factors having order
less than \( |S| \). Then given an almost simple group \( G \) with socle
\( S \), maximal \( \X \)-subgroups are exactly maximal subgroups of
\( G \) not containing \( S \). Wielandt-Hartley's theorem for
such a group \( G \) implies the following statement:
\textit{If \( G \) is a finite almost simple group with socle \( S \)
and \( M \) is a maximal subgroup of \( G \), then \( S \cap M \ne 1 \).}
It was proved by R.~A.~Wilson in \cite{Wilson} while studying novel
subgroups in almost simple groups (see also \cite[Section~1.3.1]{Bray});
and by M.~W.~Liebeck, C.~E.~Praeger and J.~Saxl in course of the proof
of the O'Nan-Scott theorem for primitive permutation groups
\cite[pp.~395--396]{LPS}.

As mentioned in Introduction, Wielandt-Hartley's theorem for submaximal \( \X \)-subgroups (Theorem~\ref{WieHartSubmax}) was announced
by Wielandt at Santa Cruz conference on finite groups in 1979
\cite[5.4(a)]{Wie3}.
At this meeting, one of the most important in the history of the classification of finite simple groups,
Wielandt gave a talk entitled "Zusammengesetzte Gruppen: H\"older Programm heute."
Concerning the subject, Wielandt anticipated (see \cite[p.~171]{Wie3})
that the theorem about submaximal \( \X \)-subgroups (in the sense of Definition~\ref{Submax} instead of Definition~\ref{StrongSubmax})
would be harder to prove and admitted that this proof had not been already written in details.
The present article provides the proof and explains why Theorem~\ref{WieHartSubmax}
is indeed stronger than Theorem~\ref{WieHartStrongSubmax}: because submaximal \( \X \)-subgroups are not always strongly submaximal.


%
%


\bigskip

\noindent
D. Revin\\
revin@math.nsc.ru\\
S. Skresanov\\
skresan@math.nsc.ru\\
A. Vasil'ev\\
vasand@math.nsc.ru\\
\medskip

\noindent
Sobolev Institute of Mathematics, 4 Acad. Koptyug avenue\\
and\\
Novosibirsk State University, 1 Pirogova street,\\
630090 Novosibirsk, Russia.\\

\end{document}